\documentclass[a4paper]{article}
\usepackage[utf8]{inputenc}
\usepackage{amsmath,amssymb,amsthm}
\usepackage[colorlinks,citecolor=blue]{hyperref}
\usepackage{mathrsfs}

\newcommand{\QQ}{\mathbb{Q}}
\newcommand{\RR}{\mathbb{R}}
\newcommand{\PP}{\mathbb{P}}
\newcommand{\Li}{\operatorname{Li}}
\newcommand{\ep}{\varepsilon}
\newcommand{\logc}{\mathscr{L}}
\newcommand{\aA}{\mathscr{A}}
\newcommand{\aB}{\mathscr{B}}
\newcommand{\aMZV}{\mathscr{A}_{\mathsf{MZV}}}
\newcommand{\aMTV}{\mathscr{A}_{\mathsf{MTV}}}
\newcommand{\aMTVc}{\mathscr{A}_{\mathsf{MTV},c}}

\newtheorem{thm}{Theorem}

\title{Multiple $T$-values with one parameter}
\author{Frédéric Chapoton}
\date{\today}

\begin{document}

\maketitle

\section*{Introduction}

Multiple zeta values can be defined as the iterated integrals of the
differential forms $dt/t$ and $dt/(1-t)$ on the real interval
$[0,1]$. They are very important and central objects, in relation with
many fields, including knot theory, quantum groups and perturbative
quantum field theory \cite{waldschmidt, fresan}. They form an algebra
$\aMZV$ over $\QQ$, which has been studied a lot, and is conjecturally
graded by the weight.

Many variants of the multiple zeta values have been considered in the
literature. Some of these variants are also numbers, for example when
allowing various roots of unity as poles of the differential
forms. Other variants are functions, with one or more arguments, such
as multiple polylogarithms \cite{goncharov} or hyperlogarithms \cite{Panzer}.

Among all these variations on a theme, a specific one,
recently introduced by Kaneko and Tsumura in \cite[\S
5]{kaneko_tsumura_1} and further studied in \cite{kaneko_final}, deals with the iterated integrals of the forms
$dt/t$ and $2dt/(1-t^2)$. The resulting numbers, called multiple
$T$-values, also form an algebra $\aMTV$ over $\QQ$ under the shuffle
product. This new algebra has been less studied than the algebra of
multiple zeta values. Conjecturally, there is an inclusion
$\aMZV \subset \aMTV$.

The present article introduces a new algebra $\aMTVc$ of iterated integrals, whose
elements are functions of a parameter $c$. This algebra can be seen as
a common deformation of both the algebras $\aMZV$ and $\aMTV$, in the
following manner. 

For every admissible index $(k_1,\dots,k_r)$, there is a
function $Z_c(k_1,\dots,k_r)$ in $\aMTVc$. These functions span the
algebra $\aMTVc$, and their product is given by the usual shuffle
rule, exactly as for multiple zeta values. One can evaluate the
function $Z_c(k_1,\dots,k_r)$ when the parameter $c$ is a real number
with $c < 1$. When $c=0$, one recovers the multiple zeta value
$\zeta(k_1,\dots,k_r)$ and when $c=-1$, one recovers the multiple
$T$-value $T(k_1,\dots,k_r)$. It follows that both $\aMZV$ and $\aMTV$
are quotient algebras of $\aMTVc$.

We start here the study of $\aMTVc$. Our original motivation was to
generalize both multiple zeta values and multiple $T$-values while
keeping the duality relations. We therefore show that for every $c$, there is an
involution of $\PP^1$ that implies the duality relations for the
functions $Z_c$. We extend a result of Kaneko and Tsumura relating a
generating function of some $Z_c$ to the hypergeometric function
${}_2F_{1}$. Using computer experiments, we determine the first few
dimensions of the graded pieces of the algebra $\aMTVc$, assuming that
it is graded by the weight. We also propose a guess for the generating
series of the graded dimensions of the algebra $\aMTV$.

It seems that the algebra $\aMTVc$ is a new object, although it is
difficult to be sure, given the very large number of articles related
to multiple zeta values and their many variants. Let us also note that a
related study can be seen in \S 4.2 of \cite{mastrolia}, where our
main differential form appears in formula (4.47).


\section{Definition and first properties}

The letter $c$ will denote a parameter, either complex or real, not
equal to $1$. In most of the article, we will assume that $c$ is real
and $c < 1$.

Consider the two differential forms:
\begin{equation}
  \omega_0(t) = \frac{dt}{t} \quad\text{and}\quad\omega_{1}(t)=\frac{(1-c)dt}{(1-t)(1-ct)}.
\end{equation}
Note that one can also write
\begin{equation}
  \label{omega1}
   \omega_{1}(t) = \frac{dt}{1-t} - \frac{c dt}{1-ct}.
\end{equation}
Because of the assumption $c < 1$, the only singularity of the
differential form $\omega_1$ on the interval $[0,1]$ is therefore the simple
pole at $1$.

When $c=0$, these two differential forms become the two differential forms
$dt/t$ and $dt/(1-t)$ whose iterated integrals are the classical
multiple zeta values (MZV). When $c=-1$, they become the two
differential forms $\Omega_0 = dt/t$ and $\Omega_1 = 2dt/(1-t^2)$ whose iterated integrals
are Kaneko-Tsumura's multiple T-values (MTV) \cite{kaneko_final}.

For general $c$, one can consider the iterated integrals of $\omega_0$
and $\omega_{1}$ as functions of the parameter $c$. We will use the definition
\begin{equation}
  \label{iterated}
  I(\ep_1, \dots, \ep_k) = \mathop{\int\cdots\int}\limits_{0<t_1<\cdots <t_k<1} \omega_{\ep_1}(t_1) \cdots \omega_{\ep_k}(t_k),
\end{equation}
where each $\ep_i$ is either $0$ or $1$, with $\ep_1=1$ and $\ep_k=0$
to ensure convergence.

Using the standard conversion of indices, let us introduce
the functions defined by
\begin{equation}
  \label{conversion}
  Z_c(k_1,\dots,k_r) = I(1,0^{k_1-1},1,0^{k_2-1},\dots,1,0^{k_r-1})
\end{equation}
for $r \geq 1$ with $k_i \geq 1$ and $k_r \geq 2$. Here powers of $0$
stand for repeated zeroes.

We have used above the same conventions as Kaneko and Tsumura in \cite{kaneko_final}, so that
the comparison with their results would be simple.

By their definition as iterated integrals, the functions $Z_c$
satisfy the same shuffle product rule as multiple zeta values and
multiple T-values. This is most easily described as a sum over the
shuffle product of indices in the notation \eqref{iterated}.  For
example,
\begin{equation}
  Z_c(2) Z_c(3) = 6 Z_c(1,4) + 3 Z_c(2,3) + Z_c(3,2).
\end{equation}
The vector space over $\QQ$ spanned by all the functions $Z_c$ is
therefore a commutative algebra, denoted by $\aMTVc$. Moreover, for any fixed $c$, the
vector space over $\QQ$ spanned by all the values of $Z_c$ is also a
commutative algebra.

As a side remark, one could wonder what happens to the relationship of
multiple zeta values with the Drinfeld associator. One could naively
replace the KZ equation by
\begin{equation}
  \frac{dF}{dz} = \left(\frac{e_0}{z} + \frac{(1-c)e_1}{(1-z)(1-cz)}\right) F
\end{equation}
and ask about properties of the solution.

\subsection{Duality}

The functions $Z_c$ also satisfy the duality property, using the
change of variable
\begin{equation}
  \label{involution}
  s = \frac{t-1}{ct-1},
\end{equation}
which makes sense as soon as $c \not= 1$. This is an involution of
$\PP^1$ that exchanges $0$ and $1$, $\infty$ with $1/c$ and
maps the interval $[0,1]$ to itself. It also exchanges $\omega_0$ and
$\omega_{1}$ up to sign:
\begin{equation}
  \label{dlogs}
  -\frac{ds}{s} = \frac{dt}{1-t} - \frac{c dt}{1-ct},
\end{equation}
which is $\omega_{1}$ by \eqref{omega1}.

This implies, with the usual proof by change of variables, that the
standard duality relations known for multiple zeta values and
multiple $T$-values also hold for the functions $Z_c$. In the notation
of \eqref{iterated}, the sequences $(\ep_1,\ep_2,\dots,\ep_k)$ and
$(1-\ep_k,\dots,1-\ep_2,1-\ep_1)$ give the same iterated
integral. This translates via \eqref{conversion} into equalities between two functions $Z_c$,
including for example
\begin{equation}
  Z_c(1,2) = Z_c(3).
\end{equation}

The unique fixed point in $[0,1]$ of the involution \eqref{involution}
is
\begin{equation}
  -\frac{\sqrt{-c + 1} - 1}{c}.
\end{equation}
This can be used for the purpose of numerical computations, as a
convenient cut-point where to apply the composition-of-paths formula
for iterated integrals.

\begin{thm}
  There is an equality between generating series:
  \begin{multline}
    1-\sum_{m,n \geq 1} Z_c(\underbrace{1,\ldots,1}_{n-1},m+1) X^m Y^n
    =\\ (1-c) \frac{\Gamma(1-X)\Gamma(1-Y)}{\Gamma(1-X-Y)} {}_2F_{1}(1-X,1-Y;1-X-Y;c).
  \end{multline}
\end{thm}
\begin{proof}
  The proof is essentially the proof of Theorem 3.6 in Kaneko and Tsumura \cite{kaneko_final}. Let us only sketch the main steps. Define an auxiliary function
  \begin{equation*}
    \logc(t) = \int_{0}^{t} \omega_1.
  \end{equation*}
  Because the integrand of $\omega_1$ is positive on $[0,1]$, the
  function $\logc$ maps as a diffeomorphism the open interval $(0,1)$
  to the positive real line $\RR_{>0}$.  From the iterated integral
  \eqref{iterated} and the conversion rule \eqref{conversion}, one
  gets by symmetrization that
  \begin{equation*}
    Z_c(\underbrace{1,\ldots,1}_{n-1},m+1) = \frac{1}{(n-1)!m!}\int_{0}^{1} \logc(t)^{n-1} \log(1/t)^m \omega_1(t).
  \end{equation*}
  Therefore
  \begin{equation*}
    \sum_{m,n \geq 1} Z_c(\underbrace{1,\ldots,1}_{n-1},m+1) X^m Y^n=
    \int_{0}^{1} e^{\logc(t) Y} (t^{-X} - 1) \omega_1(t).
  \end{equation*}
  Let us write this integral as the sum of $I_1$ and $-I_2$, where
  \begin{equation*}
    I_1 = \int_{0}^{1} e^{\logc(t) Y} t^{-X} \omega_1(t)
    = (1-c) \int_{0}^{1} \left(\frac{1-t}{1-ct} \right)^{-Y} t^{-X} \frac{dt}{(1-t)(1-ct)}
  \end{equation*}
  and
  \begin{equation*}
    I_2 = \int_{0}^{1} e^{\logc(t) Y} \omega_1(t) = \int_{0}^{\infty} e^{wY} dw= -1/Y.
  \end{equation*}
  In $I_1$, one uses the involution
  \eqref{involution}, which implies \eqref{dlogs} and
  $-\log(s) = \logc(t)$. In $I_2$, one uses the change of variables
  $w = \logc(t)$. One concludes by evaluating $I_1$ using the classical Euler integral expression
  for the hypergeometric function:
  \begin{equation*}
    \frac{\Gamma(A)\Gamma(C-A)}{\Gamma(C)} {}_2F_{1}(A,B;C;z)
    = \int_{0}^{1} t^{A-1} (1-t)^{C-A-1} (1-z t)^{-B} dt.
  \end{equation*}
\end{proof}

As for the MTV, one can expand the iterated integrals for $Z_c$ into iterated
sums using \eqref{omega1}, but this will involve not only one but
a linear combination of iterated sums. For example,
\begin{equation*}
  Z_c(2) = \int_{0<s<t<1} \omega_{1}(s) \frac{dt}{t} = \sum_{n \geq 1} \frac{1-c^n}{n^2} = \Li_2(1) - \Li_2(c).
\end{equation*}
The same argument shows that more generally
\begin{equation}
  Z_c(m) = \Li_m(1) - \Li_m(c)
\end{equation}
for all integers $m \geq 2$, where $\Li_m$ is the $m$-th polylogarithm.

\section{Dimensions}

Let us declare that the function $Z_c(k_1,\dots,k_r)$ has weight
$k_1 + \dots + k_r$. This is also the number of integration signs in the iterated integral \eqref{iterated}.

Note that it is not clear at all that only weight-homogeneous linear
relations can exist between the functions $Z_c$. This is probably
expected from the motivic philosophy, as in the classical case of MZV
\cite{brown_decomposition} and also for MTV.

Assuming that, one can then ask the following question: what are the
graded dimensions (with respect to the weight) of the vector space
$\aMTVc$ spanned by the functions $Z_c(k_1,\dots,k_r)$ over $\QQ$ ? One can
also ask the same question for any fixed value of $c$.

When $c = 0$, this question is about the algebra $\aMZV$ of MZV and
has been studied a lot. The conjecture by Zagier states that the
dimensions are the Padovan numbers. This has been proved by Brown in
the setting of motivic multiple zeta values \cite{brown_annals}.

When $c=-1$, the question has been considered in \cite{kaneko_final},
where the authors have performed large-scale computations to find the
expected dimensions in low degrees. Based on this data, one could try
to find an analogue of Zagier's conjecture. A proposal is made below
in \S \ref{conjecture_mtv}.

Here is a little table for $\aMZV$ and $\aMTV$.
\begin{equation*}
  \begin{array}{rrrrrrrrrrrrrrrrr}
          n & 0 & 1 & 2 & 3 & 4 & 5 & 6 & 7 & 8 & 9 & 10 & 11&12&13\\
          \text{MZV} & 1 & 0 & 1 & 1 & 1 & 2 & 2 & 3 & 4 & 5& 7 & 9&12&16\\
          \text{MTV} & 1 & 0 & 1 & 1 & 2 & 2 & 4 & 5 & 9 & 10&19&23&42&49\\
        \end{array}
\end{equation*}
    
It would also be interesting to consider other special values for $c$,
such as roots of unity or the inverse golden ratio, as some related variants of
multiple zeta values have already been studied. Besides $0$ and $-1$,
what are the exceptional values for $c$ where the dimensions are
smaller than those of $\aMTVc$ ?

\subsection{Constraints of duality}

\label{upper_bound}

As the functions $Z_c$ satisfy the duality relations
and form an algebra $\aMTVc$ under the shuffle product, one can try to get
purely algebraic upper bounds on the graded dimensions of this algebra.

For this, consider the shuffle algebra $\aA$ on all words in letters $0$
and $1$. The subspace $\aA'$ of $\aA$ spanned by all words starting with
$1$ and ending with $0$ is a subalgebra. One can define in $\aA'$ formal
elements $Z(k_1,\dots,k_r)$ using the conversion rule
\eqref{conversion}.

Let $\aB$ be the quotient of $\aA'$ by the ideal generated by elements
$Z(\alpha) - Z(\alpha^*)$ for all indices $\alpha$, where $*$ is the
duality of indices. This ideal contains more linear relations, the
first one being
\begin{equation}
  Z(3) (Z(3) - Z(1,2)) = 0.
\end{equation}

Using a computer, one can find the first few dimensions of $\aB$ in
small degrees:
\begin{equation*}
  \begin{array}{rrrrrrrrrrrrrrrrr}
          n & 0 & 1 & 2 & 3 & 4 & 5 & 6 & 7 & 8 & 9 & 10 & 11&12&13\\
          \aB_n & 1 & 0 & 1 & 1 & 3 & 4 & 9 & 15 & 31 & 55&109&203&397&754\\
        \end{array}.
\end{equation*}
It would be good to have some kind of formula for these dimensions.

\subsection{The case of $\aMTV$}

\label{conjecture_mtv}

In this section, we propose a guess for the dimensions of the algebra
$\aMTV$ of Kaneko-Tsumura. It would be rather bold to call this a
conjecture.

Let us start from the data given in Kaneko and Tsumura article:
\begin{equation*}
\begin{array}{rrrrrrrrrrrrrrrrr}
n & 0 & 1 & 2 & 3 & 4 & 5 & 6 & 7 & 8 & 9 & 10 & 11 & 12 & 13 & 14 & 15 \\
A & 1 & 0 & 1 & 1 & 2 & 2 & 4 & 5 & 9 & 10 & 19 & 23 & 42 & 49 & 91 & 110
\end{array}
\end{equation*}
which gives in line $A$ the conjectural dimensions obtained from
numerical experiments.

Let us add some lines to these table. First, add one line $B$ by
computing the sum of two consecutive terms in $A$, suitably
aligned. In the next line, compute the difference $B-A$ between the
lines $B$ and $A$. Next, do something rather strange, namely define
$A \sharp B$ as the column-per-column product of $A$ and $B$:
\begin{equation*}
  \begin{array}{rrrrrrrrrrrrrrrrrr}
n & 0 & 1 & 2 & 3 & 4 & 5 & 6 & 7 & 8 & 9 & 10 & 11 & 12 & 13 & 14 & 15 \\
A & 1 & 0 & 1 & 1 & 2 & 2 & 4 & 5 & 9 & 10 & 19 & 23 & 42 & 49 & 91 & 110 \\
B &   &   & 1 & 1 & 2 & 3 & 4 & 6 & 9 & 14 & 19 & 29 & 42 & 65 & 91 & 140 \\
B-A &   &   & 0 & 0 & 0 & 1 & 0 & 1 & 0 & 4 & 0 & 6 & 0 & 16 & 0 & 30 \\
A\sharp B &   &   & 1 & 1 & 4 & 6 & 16 & 30 & 81 & 140 & 361 & 667 &  & & & 
  \end{array}
\end{equation*}

Now comes the key observation: the lines $B-A$ and $A\sharp B$ seem to be the same, up to insertion of one $0$ between any two consecutive terms in $A\sharp B$.

Assuming that this equality continues to hold at every order, one can compute as many terms as one wants in the sequence $A$, using this presumed identity between $B-A$ and $A\sharp B$. This gives the sequence
\begin{multline*}
  1, 0, 1, 1, 2, 2, 4, 5, 9, 10, 19, 23, 42, 49, 91, 110, 201, 230, 431, 521, 952, 1112, 2064,\\ 2509, 4573, 5318, 9891, 12024, 21915, 25658, 47573, 57831, 105404, 122834, \dots
\end{multline*}

This is of course a rather strange procedure, which lacks an
interpretation in terms of the structure of the algebra of MTV.

In term of generating series, this can be summarized as
\begin{align*}
  A &= 1 + O(t^2),\\
  B &= (t + t^2) A - t,\\
  B &= A - 1 + t \operatorname{Diag}(A, B),
\end{align*}
where $\operatorname{Diag}(A, B)$ keeps only the diagonal terms in the product $AB$.

Another conjecture for the dimensions of $\aMTV$ has been proposed in Remark 2.2 of \cite{kaneko_final}.

\subsection{Dimensions with parameter $c$}

Let us now consider the dimensions for the algebra spanned by all
functions $Z_c$.

This section is the result of some experimental work, based on an
implementation of these functions using Pari. Linear relations were
searched in the intersection of the space of relations for MTV and
MZV, and only beyond the obvious relations deduced from duality.

Up to weight $6$, there seems to be no relation beyond the relations
that can be deduced from the duality relations. For example, in weight
$4$, there does not seem to be any relation over $\QQ$ between
$Z_c(1,3)$, $Z_c(2,2)$ and $Z_c(4)$. This implies that the graded
dimensions of $\aMTVc$ should be strictly larger than those of $\aMTV$.

In weight $7$, one finds $2$ linearly independant relations, not in
the span of relations implied by duality. So the expected dimension is
$13=15-2$. In extenso, these relations are
\begin{multline}
  Z_c(1,2,4) - 2 Z_c(1,3,3) - 4 Z_c(2,1,1,3) + 3 Z_c(2,1,4) \\
  + Z_c(2,2,3) - Z_c(2,3,2) + 2 Z_c(3,1,3) = 0
\end{multline}
and
\begin{multline}
  18 Z_c(1,1,5) + 26 Z_c(1,3,3) - 30 Z_c(1,6) + 45 Z_c(2,1,1,3) - 27 Z_c(2,1,4)\\
  - 8 Z_c(2,2,3) + 12 Z_c(2,3,2) - 15 Z_c(2,5) - 19 Z_c(3,1,3) +  Z_c(3,2,2)\\ - 4 Z_c(3,4) +  Z_c(4,1,2) = 0
\end{multline}
One can check that these $2$ relations hold exactly for MZV and
numerically for MTV.

In weight $8$, one finds $3$ linearly independant relations, not in the span of relations implied by duality. So the expected dimension is $28=31-3$. Here are these three relations as lists of coefficients:
\begin{multline*}
  [0, 0, 0, 0, 0, -1, 0, 0, -1, -3, 0, 2, 5, 10, -1, -5, 21, -2,\\
    -10, -18, 5, 0, 13, 27, -6, 52, -13, -48, 90, -45, -72]\\ 
  [0, 0, 0, -2, 0, 0, 1, -8, 0, -2, -20, 10, 6, -10, 0, -2, -16,\\
    1, 4, 6, 20, -40, 14, -8, -4, -16, 7, 24, 0, 19, 16]\\ 
  [0, 0, -5, 9, -5, 0, 0, 48, 4, 29, 135, -45, -56, -50, 4, 38, -144,\\
    7, 54, 117, -100, 270, -154, -227, 60, -479, 82, 345, -975, 366, 672] 
\end{multline*}
between the $31$ elements
\begin{multline*}
  Z_c(8), Z_c(6,2), Z_c(5,1,2), Z_c(4,4), Z_c(4,2,2),
Z_c(4,1,3), Z_c(4,1,1,2), Z_c(3,5),\\ Z_c(3,1,2,2), Z_c(3,1,1,3),
Z_c(2,6), Z_c(2,4,2), Z_c(2,3,3), Z_c(2,2,4), Z_c(2,2,2,2),\\
Z_c(2,2,1,3), Z_c(2,1,5), Z_c(2,1,3,2), Z_c(2,1,2,3), Z_c(2,1,1,4),
Z_c(2,1,1,1,3),\\ Z_c(1,7), Z_c(1,4,3), Z_c(1,3,4), Z_c(1,3,1,3),
Z_c(1,2,5), Z_c(1,2,2,3), Z_c(1,2,1,4),\\ Z_c(1,1,6), Z_c(1,1,2,4),
Z_c(1,1,1,5)
\end{multline*}
that span the space modulo the relations induced by duality. These $3$ relations also hold exactly for MZV and numerically for MTV.

In weight $9$, a similar search found $15$ relations beyond
duality. So the expected dimension would be $40 = 55-15$. This is
slightly surprising, as one may have expected to get $41 = 13+28$
from the idea of summing the two previous terms when the weight
is odd. This idea seems to work for even weights in the case of
$\aMTV$. Either our experimental work has a flaw, or this idea must be
abandoned for $\aMTVc$.

Here is a table summarizing the experimental results.
\begin{equation*}
  \begin{array}{rrrrrrrrrrrrrrrrr}
          n & 0 & 1 & 2 & 3 & 4 & 5 & 6 & 7 & 8 & 9 & 10 & 11&12&13\\
          \text{MZV} & 1 & 0 & 1 & 1 & 1 & 2 & 2 & 3 & 4 & 5& 7 & 9&12&16\\
          \text{MTV} & 1 & 0 & 1 & 1 & 2 & 2 & 4 & 5 & 9 & 10&19&23&42&49\\
          \aMTVc & 1 & 0 & 1 & 1 & 3 & 4 & 9 & 13 & 28 & 40& &&&\\
          \aB & 1 & 0 & 1 & 1 & 3 & 4 & 9 & 15 & 31 & 55&109&203&397&754\\
        \end{array}
\end{equation*}
The line labelled $\aB$ is the upper bound assuming only the relations implied by the duality relations, as explained in \S \ref{upper_bound}.

\bibliographystyle{plain}
\bibliography{article_MTVc.bib}

\end{document}